\documentclass{article}%
\usepackage{amsmath}
\usepackage{amsfonts}
\usepackage{amssymb}
\usepackage{graphicx}
\usepackage{wrapfig}
\usepackage{float}
\usepackage{xcolor}
\usepackage[cp1251]{inputenc}%
\providecommand{\U}[1]{\protect\rule{.1in}{.1in}}
\newtheorem{theorem}{Theorem}

\newtheorem{proposition}[theorem]{Proposition}

\newenvironment{proof}[1][Proof]{\noindent\textbf{#1.} }{\ \rule{0.5em}{0.5em}}
\numberwithin{equation}{section}
\begin{document}

\begin{center}
\textbf{A SIMPLE A-PRIORY ESTIMATE FOR 3D STATIONARY NAVIER-STOKES SYSTEM VIA INTERPOLATION}
\end{center}

\begin{center}
\textbf{S.P. Degtyarev, spdegt@mail.ru}
\end{center}
\begin{center}
\textbf{Moscow Technical University of Communications and Informatics,}
\end{center}
\begin{center}
\textbf{Moscow, Russian Federation}
\end{center}

\bigskip

{\small \noindent\textbf{Abstract. } This short communication is motivated by a paper by
O.A.Ladyzhenskaya, where a simple interpolation inequality was proved between
summable  smooth spaces. Such interpolation was applied as a technical tool for obtaining
estimates of the solution to the linear Stokes system. But such cooperative
simultaneous applying summable and smooth functional spaces even more often
occurs at investigations of nonlinear problems, where, in particular, some
bootstrap arguments are often involved. And we believe that the taking into
account such interpolation reasoning can notably simplify different bootstrap
procedures to rise up the smoothness of a solution or obtain an a-priory
estimate in smooth classes of functions. Therefore by the present paper we just would like
to attract the attention of the reader to such possibility. And for this we are
going to demonstrate an obtaining of the a-priory estimate in a smooth space to the
solution of the stationary Navier-Stokes system in a bounded (for simplicity)
domain. }

{\small \noindent\textbf{Keywords:} interpolation inequalities, a-priori
estimates, nonlinear PDE }

{\small \medskip}

{\small \noindent\textbf{MSC:} 26D10, 35K59 }

%

\section{Introduction}

This short communication is motivated by the paper \cite{1} by
O.A.Ladyzhenskaya, where a simple interpolation inequality was proved between
summable space $L_{2}(R^{3})$ and smooth H\"{o}lder space $C^{\alpha}(R^{3})$
(here $R^{3}$ is the $3$-dimensional Euclidian space, $\alpha\in(0,1)$). Such
interpolation was applied in \cite{1} as a technical tool for obtaining
estimates of the solution to the linear Stokes system. But such cooperative
simultaneous applying summable and smooth functional spaces even more often
occurs at investigations of nonlinear problems, where, in particular, some
bootstrap arguments are often involved. And we believe that the taking into
account such interpolation reasoning can notably simplify different bootstrap
procedures to rise up the smoothness of a solution or obtain an a-priory
estimate in smooth classes of functions.

Therefore by the present paper we just would like to attract the attention of
the reader to such possibility. And for this we are going to demonstrate an
obtaining of the a-priory estimate in the space $C^{2+\alpha}(\Omega)$ to the
solution of the stationary Navier-Stokes system in a bounded (for simplicity)
domain $\Omega$. Note that such estimate is classical and was obtained in
\cite{2} and here we just represent some another "the second step" of the
corresponding smooth estimate - when the integrable estimate is already obtained.

Let's now give several definitions and auxiliary facts. We are going to use
standard spaces $C^{m+\alpha}(\overline{\Omega})$ of H\"{o}lder continuous
functions $u(x)$, where $m=0,1,2,...$, $\alpha\in(0,1)$, $\Omega$ is a given
bounded domain in $R^{3}$ with smooth boundary $\partial\Omega$ of the class
$C^{m+\alpha}$, $\overline{\Omega}=\Omega\cup\partial\Omega$. The norms in the
space $C^{m+\alpha}(\overline{\Omega})$ is defined by%

\begin{equation}
\left\Vert u\right\Vert _{C^{m+\alpha}(\overline{\Omega})}\equiv
|u|_{\overline{\Omega}}^{(m+\alpha)}=|u|_{\overline{\Omega}}^{(m)}%
+\left\langle u\right\rangle _{\overline{\Omega}}^{(m+\alpha)}, \label{1.5}%
\end{equation}
where%

\begin{equation}
|u|_{\overline{\Omega}}^{(m)}\equiv{\sum\limits_{|\overline{\beta}|\leq m}%
}\max_{\overline{\Omega}}\left\vert D_{x}^{\overline{\beta}}u(x)\right\vert
,\quad\left\langle u\right\rangle _{\overline{\Omega}}^{(m+\alpha)}\equiv
{\sum\limits_{|\overline{\beta}|=m}}\left\langle D_{x}^{\overline{\beta}%
}u(x)\right\rangle _{\overline{\Omega}}^{(\alpha)}, \label{1.6}%
\end{equation}

\[
\left\langle w(x)\right\rangle _{\overline{\Omega}}^{(\alpha)}\equiv
\left\langle w(x)\right\rangle _{x,\overline{\Omega}}^{(\alpha)}\equiv
\sup_{x,y\in\overline{\Omega}}\frac{|w(x)-w(y)|}{|x-y|^{\alpha}},
\]
$\overline{\beta}=(\beta_{1},\beta_{2}, ..., \beta_{N})$ is a multiindex,
$\beta_{i}=0,1,2,...$, $|\overline{\beta}|=\beta_{1}+\beta_{2}+...+\beta_{N}$,%

\[
D_{x}^{\overline{\beta}}u=\frac{\partial^{\beta_{1}}\partial^{\beta_{2}%
}...\partial^{\beta_{N}}u(x)}{\partial x_{1}^{\beta_{1}}\partial x_{2}%
^{\beta_{2}}...\partial x_{3}^{\beta_{N}}}.
\]
Note that we also use the notation
\[
\left\vert u\right\vert _{\overline{\Omega}}^{(0)}\equiv\max_{\overline
{\Omega}}\left\vert u(x)\right\vert .
\]
Besides. we use Lebesgues spaces $L_{q}(\Omega)$ ,$q>1$, with the norm%
\[
\left\Vert u\right\Vert _{L_{q}(\Omega)}\equiv\left\Vert u\right\Vert
_{q,\Omega}=\left(
{\displaystyle\int\limits_{\Omega}}
\left\vert u(x)\right\vert ^{q}\right)  ^{\frac{1}{q}},
\]
and the Sobolevs spaces $W_{q}^{m}(\Omega)$ with the norm%
\[
\left\Vert u\right\Vert _{W_{q}^{m}(\Omega)}\equiv\left\Vert u\right\Vert
_{q,\Omega}^{(m)}=%
{\displaystyle\sum\limits_{k=0}^{m}}
{\displaystyle\sum\limits_{|\overline{\beta}|=k}}
\left\Vert D_{x}^{\overline{\beta}}u(x)\right\Vert _{q,\Omega}.
\]

We need also the following assertion on the interpolation.

\begin{proposition}
Let $l$ be a positive number, $l\in\lbrack0,2+\alpha)$, $\alpha\in(0,1)$, and
$q>1$. Let also $u(x)\in C^{2+\alpha}(\overline{\Omega})\cap L_{q}(\Omega)$.
Then with some $C=C(\Omega,q,l,\alpha)$
\begin{equation}
|u|_{\overline{\Omega}}^{(l)}\leq C\left(  |u|_{\overline{\Omega}}%
^{(l2+\alpha)}\right)  ^{\omega}\left(  ||u||_{q,\overline{\Omega}}\right)
^{1-\omega},\quad\omega=\frac{ql+3}{q\left(  2+\alpha\right)  +3}.
\label{2.11}%
\end{equation}
\end{proposition}
We present this Proposition without a proof. The simple proof based on the idea
from \cite{1} can be found in preprint \cite{3}, Theorem 2.

Formulate now the problem under consideration. Let a function $p(x)$ be
defined in $\overline{\Omega}$ and let $\partial p/\partial x_{i}\in
C^{\alpha}(\overline{\Omega})$, $i=1,2,3$. Let further vector-functions
$\overline{v}(x)=(v_{1}(x),v_{2}(x),v_{3}(x))$ and $\overline{f}%
(x)=(f_{1}(x),f_{2}(x),f_{3}(x))$\ be defined in $\overline{\Omega}$ and let
$\overline{v}(x)\in C^{2+\alpha}(\overline{\Omega})$, $\overline{f}(x)\in
C^{\alpha}(\overline{\Omega})$. Let at last $p(x)$ and $\overline{v}(x)$
satisfy the system of equations%
\begin{equation}
-\nu\Delta\overline{v}(x)+\left(  \overline{v}(x)\nabla\right)  \overline
{v}(x)+\nabla p(x)=\overline{f}(x),\quad x\in\Omega, \label{1.1}%
\end{equation}%
\begin{equation}
\nabla\overline{v}(x)=0,\quad x\in\Omega, \label{1.2}%
\end{equation}%
\begin{equation}
\overline{v}(x)=0,\quad x\in\partial\Omega. \label{1.3}%
\end{equation}
Here $\nu=const>0$, $\Delta=(\partial^{2}/\partial x_{1}^{2}+\partial
^{2}/\partial x_{2}^{2}+\partial^{2}/\partial x_{3}^{2})$ is the Laplace
operator, $\nabla=(\partial/\partial x_{1},\partial/\partial x_{2}%
,\partial/\partial x_{3})$, so that
\[
\overline{v}(x)\nabla=v_{1}(x)\frac{\partial}{\partial x_{1}}+v_{2}%
(x)\frac{\partial}{\partial x_{2}}+v_{3}(x)\frac{\partial}{\partial x_{3}%
},\quad\nabla\overline{v}(x)=%
{\displaystyle\sum\limits_{i=1}^{3}}
\frac{\partial v_{i}(x)}{\partial x_{i}}.
\]

By the generalized solution from $W_{2}^{1}(\Omega)$ to problem
\eqref{1.1}-\eqref{1.3} we mean (\cite{2}) functions $\overline{v}(x)\in
W_{2}^{1}(\Omega)$ and $p(x)$ such that $\overline{v}(x)$ satisfies the
integral identity%
\begin{equation}
\nu%
{\displaystyle\int\limits_{\Omega}}
{\displaystyle\sum\limits_{i,k=1}^{3}}
\frac{\partial v_{i}}{\partial x_{k}}\frac{\partial\eta_{i}}{\partial x_{k}%
}dx+%
{\displaystyle\int\limits_{\Omega}}
{\displaystyle\sum\limits_{i,l=1}^{3}}
v_{l}\frac{\partial v_{i}}{\partial x_{l}}\eta_{i}(x)dx=%
{\displaystyle\int\limits_{\Omega}}
{\displaystyle\sum\limits_{i=1}^{3}}
\eta_{i}(x)f_{i}(x)dx \label{1.8}%
\end{equation}
for an arbitrary vector-function $\overline{\eta}(x)=(\eta_{1}(x),\eta
_{2}(x),\eta_{3}(x))\in W_{2}^{1}(\Omega)$ with $\nabla\overline{\eta
}(x)\equiv0$ in $\Omega$, and $p(x)$ is defined up to a constant from
\eqref{1.1} over $\overline{v}(x)$ in the distributional sense. Note that
since $\overline{v}(x)$ and $p(x)$ are smooth and satisfy \eqref{1.1}-\eqref{1.3}
in the classical sense, they are a generalized solution as well.

Due to \cite{2} we have the following assertion (see \cite{2}, Chapter 5, Theorem 4).
\begin{proposition}
Under above assumption $\overline{f}(x)\in C^{\alpha}(\overline{\Omega})$
problem \eqref{1.1}-\eqref{1.3} has at least one generalized solution
$\overline{v}(x)$, $p(x)$ with $\overline{v}(x)\in W_{2}^{1}(\Omega)$.
\end{proposition}
We suppose further that the weak $W_{2}^{1}(\Omega)$ - norm  of $\overline{v}(x)$ is somehow estimated
already and
\begin{equation}
|| \overline{v}(x)|| _{2,\Omega}^{(1)}\leq C(\Omega
,\max_{\overline{\Omega}}\left\vert \overline{f}(x)\right\vert ). \label{1.4}%
\end{equation}

On the ground of the well known the Sobolev embedding theorem in our $3d$ space dimension  we can infer
from this proposition the following embedding
\begin{equation}
|| \overline{v}(x)|| _{6,\Omega}\leq C|| \overline
{v}(x)||_{2,\Omega}^{(1)},\quad\text{since}\quad \frac
{2}{3} - \frac{1}{2}=\frac{1}{6},\quad6=2^{\ast}=\frac{3\cdot2}{3-2}. \label{1.7}%
\end{equation}
Now we are ready to obtain an a-priory estimate of a solution to
\eqref{1.1}-\eqref{1.3} in the smooth space $C^{2+\alpha}(\Omega)$ and this
will be done in the next section.

\section{A-priory estimate}

We have the following main assertion.
\begin{theorem}
If in problem \eqref{1.1}-\eqref{1.3} $\overline{f}(x)\in C^{\alpha}%
(\overline{\Omega})$, then any solution $\overline{v}(x)$, $p(x)$ to this
problem from the space $\overline{v}(x)\in C^{2+\alpha}(\overline{\Omega})$,
$\nabla p(x)\in C^{\alpha}(\overline{\Omega})$ obey the estimate 
\[
\left\vert \overline{v}(x)\right\vert _{\overline{\Omega}}^{(2+\alpha)}\leq
C\left\vert \overline{f}(x)\right\vert _{\overline{\Omega}}^{(\alpha
)}+C\left(  \left\Vert \overline{v}(x)\right\Vert _{2,\Omega}^{(1)}\right)
^{B}, \ \ \ B=\frac{2 - a_{1}}{1 - a_{1}}, \ \ a_{1}=\frac{6\left(  2+\alpha\right)  }{6\left(  2+\alpha\right)  +3}.
\]
\end{theorem}
\begin{proof}
First of all,  since we consider a solution to \eqref{1.1}-\eqref{1.3} from
the smooth class, this solution is of course a generalized solution in the
sense of \eqref{1.8} with
\begin{equation}
\overline{v}(x)\in W_{2}^{1}(\Omega),\quad\left\Vert \overline{v}%
(x)\right\Vert _{2,\Omega}^{(1)}<\infty,\quad\left\Vert \overline
{v}(x)\right\Vert _{6,\Omega}\leq C\left\Vert \overline{v}(x)\right\Vert
_{2,\Omega}^{(1)}<\infty.\label{1.9}%
\end{equation}
Consider now the nonlinear term $\left(  \overline{v}(x)\nabla\right)
\overline{v}(x)$ from \eqref{1.1} in the space $C^{\alpha}(\overline{\Omega}%
)$. Note that%
\begin{equation}
\left\vert \left(  \overline{v}(x)\nabla\right)  \overline{v}(x)\right\vert
_{\overline{\Omega}}^{(0)}\leq%
{\displaystyle\sum\limits_{i,k,l=1}^{3}}
\left\vert v_{i}\right\vert _{\overline{\Omega}}^{(0)}\left\vert
\frac{\partial v_{k}}{\partial x_{l}}\right\vert _{\overline{\Omega}}%
^{(0)},\label{1.10}%
\end{equation}%
\begin{equation}
\left\langle \left(  \overline{v}(x)\nabla\right)  \overline{v}%
(x)\right\rangle _{\overline{\Omega}}^{(\alpha)}\leq%
{\displaystyle\sum\limits_{i,k,l=1}^{3}}
\left\vert v_{i}\right\vert _{\overline{\Omega}}^{(0)}\left\langle
\frac{\partial v_{k}}{\partial x_{l}}\right\rangle _{\overline{\Omega}%
}^{(\alpha)}+%
{\displaystyle\sum\limits_{i,k,l=1}^{3}}
\left\langle v_{i}\right\rangle _{\overline{\Omega}}^{(\alpha)}\left\vert
\frac{\partial v_{k}}{\partial x_{l}}\right\vert _{\overline{\Omega}}%
^{(0)}.\label{1.11}%
\end{equation}
Bearing in mind \eqref{1.9}, we now apply \eqref{2.11} with $q=6$ and with
$l=0,\alpha,1,1+\alpha$ to each term in in the two above inequalities to
obtain the estimates%
\begin{equation}
\left\vert v_{i}\right\vert _{\overline{\Omega}}^{(0)}\leq C\left(
|v_{i}|_{\overline{\Omega}}^{(l2+\alpha)}\right)  ^{\omega_{0}}\left(
||v_{i}||_{6,\overline{\Omega}}\right)  ^{1-\omega_{0}},\quad\omega_{0}%
=\frac{3}{6\left(  2+\alpha\right)  +3},\label{1.12}%
\end{equation}%
\begin{equation}
\left\langle v_{i}\right\rangle _{\overline{\Omega}}^{(\alpha)}\leq\left\vert
v_{i}\right\vert _{\overline{\Omega}}^{(\alpha)}\leq C\left(  |v_{i}%
|_{\overline{\Omega}}^{(l2+\alpha)}\right)  ^{\omega_{\alpha}}\left(
||v_{i}||_{6,\overline{\Omega}}\right)  ^{1-\omega_{\alpha}},\quad
\omega_{\alpha}=\frac{6\alpha+3}{6\left(  2+\alpha\right)  +3},\label{1.12+1}%
\end{equation}%
\begin{equation}
\left\vert \frac{\partial v_{k}}{\partial x_{l}}\right\vert _{\overline
{\Omega}}^{(0)}\leq\left\vert v_{k}\right\vert _{\overline{\Omega}}^{(1)}\leq
C\left(  |v_{k}|_{\overline{\Omega}}^{(l2+\alpha)}\right)  ^{\omega_{1}%
}\left(  ||v_{k}||_{6,\overline{\Omega}}\right)  ^{1-\omega_{1}},\quad
\omega_{1}=\frac{6\cdot1+3}{6\left(  2+\alpha\right)  +3},\label{1.14}%
\end{equation}%
\begin{equation}
\left\langle \frac{\partial v_{k}}{\partial x_{l}}\right\rangle _{\overline
{\Omega}}^{(\alpha)}\leq\left\vert v_{k}\right\vert _{\overline{\Omega}%
}^{(1+\alpha)}\leq C\left(  |v_{k}|_{\overline{\Omega}}^{(l2+\alpha)}\right)
^{\omega_{1+\alpha}}\left(  ||v_{k}||_{6,\overline{\Omega}}\right)
^{1-\omega_{1+\alpha}},\quad\omega_{1+\alpha}=\frac{6(1+\alpha)+3}{6\left(
2+\alpha\right)  +3}.\label{1.15}%
\end{equation}
Now from \eqref{1.10}-\eqref{1.15} we have%
\[
\left\vert \left(  \overline{v}(x)\nabla\right)  \overline{v}(x)\right\vert
_{\overline{\Omega}}^{(\alpha)}=\left\vert \left(  \overline{v}(x)\nabla
\right)  \overline{v}(x)\right\vert _{\overline{\Omega}}^{(0)}+\left\langle
\left(  \overline{v}(x)\nabla\right)  \overline{v}(x)\right\rangle
_{\overline{\Omega}}^{(\alpha)}\leq
\]%
\[
\leq\left\vert \overline{v}\right\vert _{\overline{\Omega}}^{(0)}\left\vert
\overline{v}\right\vert _{\overline{\Omega}}^{(1+\alpha)}+\left\vert
\overline{v}\right\vert _{\overline{\Omega}}^{(\alpha)}\left\vert \overline
{v}\right\vert _{\overline{\Omega}}^{(1)}\leq
\]%
\begin{equation}
\leq C\left[  \left(  |\overline{v}|_{\overline{\Omega}}^{(l2+\alpha)}\right)
^{\omega_{0}+\omega_{1+\alpha}}\left(  ||\overline{v}||_{6,\overline{\Omega}%
}\right)  ^{2-\omega_{0}-\omega_{1+\alpha}}+\left(  |\overline{v}%
|_{\overline{\Omega}}^{(l2+\alpha)}\right)  ^{\omega_{\alpha}+\omega_{1}%
}\left(  ||\overline{v}||_{6,\overline{\Omega}}\right)  ^{2-\omega_{\alpha
}-\omega_{1}}\right]  =\label{2.16}%
\end{equation}%
\[
=C\left(  |\overline{v}|_{\overline{\Omega}}^{(l2+\alpha)}\right)  ^{a_{1}%
}\left(  ||\overline{v}||_{6,\overline{\Omega}}\right)  ^{a_{2}},
\]
where
\begin{equation}
a_{1}=\frac{6\left(  2+\alpha\right)  }{6\left(  2+\alpha\right)  +3}<1,\quad
a_{2}=2-a_{1}.\label{2.17}%
\end{equation}
Since $a_{1}<1$ in \eqref{2.16}, we can apply the Young inequality with
$\varepsilon$ to proceed with estimate \eqref{2.16} as%
\begin{equation}
\left\vert \left(  \overline{v}(x)\nabla\right)  \overline{v}(x)\right\vert
_{\overline{\Omega}}^{(\alpha)}\leq\varepsilon|\overline{v}|_{\overline
{\Omega}}^{(l2+\alpha)}+C\varepsilon^{-A}\left(  ||\overline{v}||_{6,\overline
{\Omega}}\right)  ^{B},\label{2.18}%
\end{equation}
where%
\begin{equation}
A=\frac{a_{1}}{(1-a_{1})a_{2}},\quad B=\frac{a_{2}}{1-a_{1}}.\label{2.19}%
\end{equation}
Moving now the term $\left(  \overline{v}(x)\nabla\right)  \overline{v}(x)$
to the right hand side of \eqref{1.1}, we can rewrite \eqref{1.1}-\eqref{1.3}
as the linear problem%
\begin{equation}
-\nu\Delta\overline{v}(x)+\nabla p(x)=\overline{g}(x)=\overline{f}(x)-\left(
\overline{v}(x)\nabla\right)  \overline{v}(x),\quad x\in\Omega,\label{2.20}%
\end{equation}%
\begin{equation}
\nabla\overline{v}(x)=0,\quad x\in\Omega,\label{2.21}%
\end{equation}%
\begin{equation}
\overline{v}(x)=0,\quad x\in\partial\Omega,\label{2.22}%
\end{equation}
where%
\begin{equation}
\left\vert \overline{g}(x)\right\vert _{\overline{\Omega}}^{(\alpha)}%
\leq\varepsilon|\overline{v}|_{\overline{\Omega}}^{(l2+\alpha)}+\left\vert
\overline{f}(x)\right\vert _{\overline{\Omega}}^{(\alpha)}+C\varepsilon
^{-A}\left(  ||\overline{v}||_{6,\overline{\Omega}}\right)  ^{B}.\label{2.23}%
\end{equation}

Applying to \eqref{2.20}-\eqref{2.22} Theorem 4 from Chapter 3 in \cite{2}, we
obtain%
\begin{equation}
\left\vert \overline{v}(x)\right\vert _{\overline{\Omega}}^{(2+\alpha
)}+\left\vert \nabla p\right\vert _{\overline{\Omega}}^{(\alpha)}\leq
C\varepsilon|\overline{v}|_{\overline{\Omega}}^{(l2+\alpha)}+C\left\vert
\overline{f}(x)\right\vert _{\overline{\Omega}}^{(\alpha)}+C\varepsilon
^{-A}\left(  ||\overline{v}||_{6,\overline{\Omega}}\right)  ^{B}.\label{2.24}%
\end{equation}
Choosing here $\varepsilon$ sufficiently small ($C\varepsilon<1$) and
absorbing the term $C\varepsilon|\overline{v}|_{\overline{\Omega}}%
^{(l2+\alpha)}$ in the left hand side we arrive at%
\[
\left\vert \overline{v}(x)\right\vert _{\overline{\Omega}}^{(2+\alpha)}\leq
C\left\vert \overline{f}(x)\right\vert _{\overline{\Omega}}^{(\alpha
)}+C\left(  ||\overline{v}||_{6,\overline{\Omega}}\right)  ^{B}\leq
C\left\vert \overline{f}(x)\right\vert _{\overline{\Omega}}^{(\alpha
)}+C\left(  \left\Vert \overline{v}(x)\right\Vert _{2,\Omega}^{(1)}\right)
^{B},
\]
which finishes the proof.
\end{proof}

\end{document}